\documentclass[10pt,twoside,a4paper,reqno]{amsart}
\usepackage{amscd,amsmath,amsthm,amsfonts,latexsym,amssymb}

\theoremstyle{plain}
\newtheorem{theo+}           {Theorem}
\newtheorem{prop+}           {Proposition}
\newtheorem{coro+}           {Corollary}
\newtheorem{lemm+}           {Lemma}

\theoremstyle{definition}
\newtheorem{defi+}           {Definition}
\newtheorem{exam+}           {Example}
\newtheorem{problem}         {Problem}
\newtheorem*{question}       {Question}
\newtheorem*{ack}            {Acknowledgements}

\theoremstyle{remark}
\newtheorem{rema+}           {Remark}

\newenvironment{theorem}{\begin{theo+}}{\end{theo+}}

\newenvironment{lemma}{\begin{lemm+}}{\end{lemm+}}
\newenvironment{remark}{\begin{rema+}}{\end{rema+}}

\newenvironment{example}{\begin{exam+}}{\end{exam+}}

\newcommand {\cP} {\mathbb {CP}}
\newcommand {\bC} {\mathbb {C}}

\newcommand {\al} {\alpha}
\newcommand {\be} {\beta}

\newcommand {\la} {\lambda}
\newcommand {\eps} {\epsilon}

\newcommand \M {\mathcal {M}}
\newcommand \E {\mathcal E}
\newcommand \PP {\mathcal P}
\newcommand \LL {\mathcal L}
\newcommand \fL {\mathfrak L}

\newcommand \C {\mathcal C}
\newcommand \A {\mathcal A}
\newcommand \HP {\mathcal{HP}}

\begin{document}

\numberwithin{equation}{section}

\title[Eigenvalues of rectangular matrices]{On eigenvalues of rectangular
matrices}
\author[J.~Borcea]{Julius Borcea}
\address{Department of Mathematics, Stockholm University, SE-106 91 Stockholm,
   Sweden}
\email{julius@math.su.se}
\author[B.~Shapiro]{Boris Shapiro}
\address{Department of Mathematics, Stockholm University, SE-106 91 Stockholm,
   Sweden}
\email{shapiro@math.su.se}
\author[M.~Shapiro]{Michael Shapiro}
\address{Department of Mathematics, Michigan State University, East
Lansing, MI 48824-1027, USA}
\email{mshapiro\@math.msu.edu}


\begin{abstract}
Given a $(k+1)$-tuple  $A, B_1,\ldots, B_k$  of $(m\times n)$-matrices  with $m\le n$ we call the set of all $k$-tuples  of complex numbers $\{\la_1,\ldots,\la_k\}$ such that the linear combination $A+\la_1B_1+\la_2B_2+\ldots+\la_kB_k$ has rank smaller than $m$ the {\it eigenvalue locus} of the latter pencil.  Motivated primarily by applications to multi-parameter generalizations of the Heine-Stieltjes spectral problem, see \cite {He} and  \cite {Vol}, we study   a number of  properties of the eigenvalue locus in the most important  case $k=n-m+1$.
\end{abstract}

\maketitle







\section*{Introduction and Main Results}  In recent years there appeared a  number of publications discussing the eigenvalues  of pencils of non-square matrices and their approximations, see, e.g., \cite{BEGM}, \cite{CG}, \cite{TW} and references therein.  But to the best of our knowledge the following natural problem either has been overlooked by specialists in linear algebra or is deeply buried in the (enormous) literature on this topic.

\begin{question}
Given a $(k+1)$-tuple of $(m\times n)$-matrices $A, B_1,\ldots, B_k$, $m\le n$, describe the set of all values of the parameters $\la_1,\ldots\la_k$ for which the rank of the linear combination $A+\la_1B_1+\ldots+\la_kB_k$ is less than $m$ or, in other words, when the linear system   $v*(A+\la_1B_1+\ldots+\la_kB_k)=0$¾  has a nontrivial left solution $0\neq v\in \bC^m$ which we call an {\em eigenvector}, where the symbol ``$*$'' denotes the usual matrix/vector multiplication.
\end{question}

Let  $\M(m,n)$, $m\le n$, be the linear space  of all $(m\times n)$-matrices with complex entries. In what follows we will consider $k$-tuples of $(m\times n)$-matrices $B_1,\ldots, B_k$ which are linearly independent in $\M(m,n)$ and denote their linear span by $\LL=\LL(B_1,\ldots,B_k)$.
Given a matrix pencil   $\PP=A+\LL$, where $A\in \M(m,n)$, let $\E_\PP\subset \PP$ be its {\it eigenvalue locus}, i.e., the set of matrices in $\PP$ whose rank is less than $m$.  Elements of $\E_\PP$ will be called {\it (generalized) eigenvalues.}  Denote by $\M^1\subset \M(m,n)$ the set of all $(m\times n)$-matrices with positive corank, i.e., whose rank is non-maximal.
Its codimension equals $n-m+1$ and its degree as an algebraic variety equals $\binom {n}{m-1}$,  see
\cite[Proposition 2.15]{BV}.  Consider the natural left-right action of the group $GL_m\times GL_n$ on $\M(m,n)$, i.e.,  $GL_m$ (respectively, $GL_n$) acts on $(m\times n)$-matrices by left (respectively, right) multiplication. This action on  $\M(m,n)$ has finitely many orbits, each orbit being the set of all matrices of a given (co)rank, see, e.g., \cite[Chap.~I \S 2]{AVG}. Note that by the well-known product formula for coranks the codimension of the set of matrices of rank $r$ equals $(m-r)(n-r)$.  Obviously, for any pencil $\PP$ one has that the eigenvalue locus coincides with $\E_\PP=\M^1\cap \PP$. Thus for a generic pencil $\PP$ of dimension $k$ the eigenvalue locus $\E_\PP$ is a subvariety of $\PP$ of codimension $n-m+1$ if $k\ge n-m+1$ and it is empty otherwise. The most interesting situation for applications occurs  when $k=n-m+1$, in which case $\E_\PP$ is generically a finite set. From now on we assume that $k=n-m+1$. Denoting as above  by $\LL$ the linear span of
$B_1,\ldots,B_{n-m+1}$ we say that $\LL$ is {\it transversal to $\M^1$} if the intersection $\LL\cap \M^1$ is finite and {\it non-transversal to $\M^1$}¾  otherwise.  Notice that due to the homogeneity of $\M^1$ any  $(n-m+1)$-dimensional linear subspace $\LL$ transversal to it   intersects $\M^1$ only at $0$ and that the multiplicity of this intersection at $0$  equals $\binom{n}{m-1}$. 

An important  and most natural example of such  a subspace $\LL$ is motivated by the Heine-Stieltjes theory \cite{He} and its higher order generalizations \cite{BBS}.  Denote by $J_s$, $s=1,\ldots,n-m+1$, the $(m\times n)$-matrix whose entries are given by $a_{i,j}=0$ if $i-j\neq s$ and $1$ otherwise. We call $J_s$ the {\it  $s$-th unit matrix} or the {\it  the $s$-th diagonal matrix}. Let us denote  the linear span  of $J_1,\ldots,J_{n-m+1}$ by $\fL$ and call $\fL$ the {\it standard diagonal subspace}. Note that $\fL$ is transversal to $\M^1$  since any matrix in $\fL$ different from $0$ has full rank, as one can easily check.

We start with the  following simple statement.

\begin{lemma}\label{lm:int}
If $\LL\subset \M(m,n)$ has dimension $(n-m+1)$ and is tranversal to $\M^1$ then for any matrix $A\in \M(m,n)$ the eigenvalue locus $\E_\PP$ of the pencil $\PP=A+\LL$ consists of exactly $\binom {n}{m-1}$ points counted with multiplicities.
\end{lemma}

\begin{remark}
Notice that since $\M^1$ is an incomplete intersection the same holds for the eigenvalue
locus  $\E_\PP$ of a generic pencil $\PP=A+\LL$, i.e., in order to find $\E_\PP$ for a given generic matrix $A$ and a given generic subspace $\LL$ one has to solve an overdetermined system of  determinantal  equations.  
\end{remark}

However, as was  essentially discovered  by Heine \cite{He}, the situation is different if one considers the standard diagonal subspace $\fL$ and any $A=(a_{i,j})\in \M(m,n)$ which is upper-triangular -- that is, such that $a_{i,j}=0$ whenever $i>j$ -- and has additionally distinct elements on the first main diagonal.

\begin{theorem}\label{th:Heine} For any upper-triangular matrix $A=(a_{i,j})\in \M(m,n)$ with all distinct  entries $a_{i,i}$ on the first main diagonal the
eigenvalue locus $\E_\PP$ of the pencil $\PP=A+\fL$, where $\fL$ is the standard diagonal subspace, is the union of $m$ complete intersections enumerated by the first component of the eigenvalue.
\end{theorem}

\begin{remark}
An explicit defining  system of $(n-m)$ algebraic equations in $(n-m)$ variables for each such complete intersection is presented in the proof of Theorem \ref{th:Heine}, see \S \ref{s1} below.
\end{remark}

Given $\LL$ as above consider the natural projection map $\pi_\LL:\M(m,n)\to \LL^\perp$ along $\LL$,  where  $\LL^\perp=\M(m,n)/\LL$. Noticing that $\dim \M^1=\dim \LL^\perp$ we define the {\it set of critical values} of $\pi_\LL$ to be 
the set $\C_\LL$ of all points in $\M^1$ where $\pi_\LL$ is not a local diffeomorphism of
$\M^1$ on its image  $\pi_\LL(\M^1)$. In other words, $\C_\LL$ is the set of all points $p\in \M^1$ such that the sum of $\LL$  and the tangent space to $\M^1$ at $p$  does not coincide with the whole $\M(m,n)$. In particular,  independently of $\LL$ the critical value set $\C_\LL$ always includes the set $\M^2$ of all $(m\times n)$-matrices with corank at least $2$.

Recall that $\M^1\subset \M(m,n)$ has the  classical small resolution of singularities $\widetilde{\M^1}\subset \M(m,n)\times \cP^{m-1}$. Here $\widetilde{\M^1}$ consists of all pairs $\left(A,pker(A)\right)$, where $A\in \M^1$ and $pker(A)$ is the projectivization of the left kernel of $A$.  Using this construction one can parameterize a Zariski open subset of $\M^1$ as follows. Consider the product
$P(m,n)= \M(m-1,n)\times \bC^{m-1}$. Take the map $\nu: P(m,n)\to \M^1\subset \M(m,n)$ sending
a pair $(\A; k_1,\ldots,k_{m-1})$ to the matrix $A\in \M(m,n)$ obtained by appending to $\A$ the last row such that its sum with  the linear combination with the coefficients $(k_1,\ldots,k_{m-1})$ of the respective rows of $\A$ vanishes.

The  main result of this paper is a simple determinantal representation of $\C_\LL$ in the above coordinates.

\begin{theorem} \label{th:determ} Let $\LL$ be any  $(n-m+1)$-dimensional linear subspace  in $\M(m,n)$ transversal to $\M^1$ and denote by  $L_1,\ldots,L_{n-m+1}$ some  basis of  $\LL$. Then  in the coordinates of $P(m,n)$ the critical value set  $\C_\LL$¾ is given the determinantal equation
\begin{equation}\label{eq:1}
 \det\begin{pmatrix} \A\\
                                V_1\\
                                \vdots\\
                                V_{n-m+1}
 \end{pmatrix}=0.\end{equation}
 Here $\A$ is a $(m-1,n)$-matrix with undetermined entries and $V_j$, $j=1,\ldots,n-m+1$, are row vectors given by $V_j=\mathbf{\kappa}* L_j$, where $\mathbf{\kappa}=(k_1,\ldots,k_{m})$.
\end{theorem}

\begin{remark}
If one expands equation (\ref{eq:1}) in the variables $(k_1,\ldots,k_{n-m+1})$ then the coefficient of each monomial in these variables is a linear combination of the maximal minors of $\A$ 
 (i.e., the Pl\"ucker coordinates) with complex coefficients depending only on the choice of $\LL$.
Moreover, the above equation contains a lot of information of geometric nature.
\end{remark}

Our next result shows that for the standard diagonal subspace $\fL$ the determinantal equation in Theorem \ref{th:determ} can be made quite a bit more explicit, which is particularly convenient from a computational viewpoint. We need first some additional notation. If $s\ge 1$ is an integer and $1\le r\le s$ let $Q_{r,s}$ be the set of all strictly increasing sequences of $r$ integers chosen from $1,\ldots,s$. Note in particular that $Q_{s,s}$ consists of a single sequence, namely $\{1,\ldots,s\}$. For $\al=(\al_1,\ldots,\al_r)\in Q_{r,s}$ set $\rho(\al)=\sum_{j=1}^{r}\al_j$. Given $A\in\M(m,n)$, $1\le k\le m$, $1\le l\le n$, $\al\in Q_{k,m}$ and $\be\in Q_{l,n}$ denote by $A[\al|\be]\in\M(k,l)$ the submatrix of $A$ lying in rows $\al$ and columns $\be$.
Let $\HP(i,d)$ denote the complex space of all homogeneous polynomials in $i$ variables of degree $d$ and define the $d\times (i+d-1)$ matrix
$$T_{i,d}=T_{i,d}(k_1,\ldots,k_i)=k_1J_1+\ldots+k_iJ_i,$$
where $J_j$, $1\le j\le i$, is as before the $j$-th diagonal $d\times (i+d-1)$
matrix and $k_1,\ldots,k_i$ are indeterminates. We will also need a result that may be of independent interest, namely the following lemma.

\begin{lemma}\label{l:top}
In the above notation, the $\binom{i+d-1}{d}$ polynomials in $k_1,\ldots,k_i$ given by the determinants
$$\Big|T_{i,d}[\al|\be]\Big|,\quad \al\in Q_{d,d},\,\be\in Q_{d,i+d-1},$$
build a basis of $\HP(i,d)$.
\end{lemma}

\begin{remark}
The usual determinant expansion formula provides an explicit expression (albeit tedious and not really needed for the present purposes) for the
$\binom{i+d-1}{d}\times\binom{i+d-1}{d}$ matrix relating the standard monomial basis of $\HP(i,d)$ to the one constructed in Lemma \ref{l:top}.
\end{remark}

\begin{theorem}\label{th:sds}
Let $\A\in\M(m-1,n)$ be as in Theorem \ref{th:determ}. The homogeneous defining polynomial of $\C_\fL$ with respect to the standard diagonal subspace $\fL$ is given by
$$\sum_{\be\in Q_{m-1,n}}(-1)^{\rho(\be)}\Big|\A\big[\{1,\ldots,m-1\}|\be\big]\Big|\cdot
\Big|T_{m,n-m+1}\big[\{1,\ldots,n-m+1\}|\{1,\ldots,n\}\setminus\be\big]\Big|.$$
\end{theorem}

\begin{example} For $m=2$ the homogeneous defining polynomial of $\C_\fL$ with
respect to the standard diagonal subspace $\fL$ is given by 
$$a_nk_1^n+a_{n-1}k_1^{n-1}k_2+a_{n-2}k_1^{n-2}k_2^2+\ldots+a_1k_1k_2^{n-1},$$
where $a_j=a_{1,j}$, $j=1,\ldots,n$.
\end{example}

\begin {example} For the standard diagonal  subspace $\fL$ in the case of  $\M(3,4)$ the homogeneous defining polynomial of $\C_\fL$ may be written as
\begin{multline*}
\begin{vmatrix} a_{1,1}&a_{1,2}&a_{1,3}&a_{1,4}\\
                                a_{2,1}&a_{2,2}&a_{2,3}&a_{2,4}\\
                                k_1&k_2&k_3&0\\
                                0&k_1&k_2&k_3
 \end{vmatrix}=\Delta_{3,4}k_1^2+\Delta_{1,4}k_2^2+\Delta_{1,2}k_3^2-\Delta_{2,4}k_1k_2\\
+(\Delta_{2,3}-\Delta_{1,4})k_1k_3-\Delta_{1,3}k_2k_3,
\end{multline*}
 where $\Delta_{i,j}$ is the $(2\times 2)$-determinant of the upper part $\A$ including the $i$-th and $j$-th columns.
\end{example}


%
%


\begin{remark}
The multiplicity of an eigenvalue $A\in \E_\PP$ can be  expressed in terms of the
dimension of the corresponding local algebra. More exactly,  for an $(m\times n)$-matrix $A$ we
define the ideal $I_A$ in the algebra
$\bC[[t_1,\dots,t_k]]$ of formal power series as the ideal generated
by all Pl\"ucker polynomials $\Delta_{i_1,\dots,i_m}(A+\sum_{l=1}^m
t_lB_l)$, where $\Delta_{i_1,\dots,i_m}(X_{m\times n})$ is the 
determinant of the $(m\times m)$ submatrix of X formed by the columns 
with the indices $i_1, i_2,\dots,i_m$. Now define the local algebra $\A_{loc}$ as the quotient algebra
$A_{loc}=\bC[[t_1,\dots,t_k]]/I_A$. Then the multiplicity of the eigenvalue $A$ in the pencil $A+\LL$
equals $\dim_\bC A_{loc}$.
\end{remark}

The main result of this note (Theorem~\ref{th:determ}) gives a simple explicit determinantal formula for the critical value set $\C_\LL$ (in coordinates on the resolution of singularities $P(m,n)$). Its inverse image $\pi_\LL^{-1}(\C_\LL)$ is an important hypersurface  consisting of all matrices in $\M(m,n)$ having a multiple eigenvalue. However, the problem of obtaining explicitly its defining polynomial in matrix entries seems to be quite delicate in general. As an illustration, let us show how this can be done in the simplest case of $(2\times 3)$-matrices.    

\begin{example}[Discriminant equation]
For $m=2, n=3$ we will write the
defining  equation for the hypersurface $\pi_\LL^{-1}(\C_\LL)$  of matrices with multiple
eigenvalues in the space $\M(2,3)$ itself.

For any pair of positive integers $m<n$ consider the extended matrix space $\M(m,n)\times \bC
P^{m-1}\times \bC^{n-m+1}$,
 where the $m$-tuple of homogeneous coordinates in $\bC^m$ is denoted by
$\kappa=(\kappa_1:\dots:\kappa_m)$ and the coordinates  in
$\bC^{n-m+1}$ are denoted by
$\lambda=(\lambda_1,\dots,\lambda_{n-m+1})$.

Given a matrix $M\in \M(m,n)$ we will write a system of polynomial
equations for 
$A+\sum_{i=1}^{n-m+1}\lambda_i J_i$, 
\[\sum_{s=1}^m \kappa_s [A+\sum_{i=1}^{n-m+1}\lambda_i
J_i]_{s*}=0,\]
expressing the fact that $\lambda$ is an eigenvalue of
$M$ while the $\kappa_i$'s are the corresponding coefficients of a
linear dependence between the rows of the matrix.

Using resultants we can  get rid of the additional variables $\lambda$ and
$\kappa$. This elimination  leads to the defining equation for the hypersurface in question.

Namely,  consider a $(2\times 3)$-matrix
$A=\left(%
\begin{array}{ccc}
  a_{11} & a_{12} & a_{13} \\
  a_{21} & a_{22} & a_{23} \\
\end{array}%
\right)$ and let as before 
$J_1=\left(%
\begin{array}{ccc}
  1 & 0 & 0 \\
  0 & 1 & 0 \\
\end{array}\right)$
and $J_2=\left(%
\begin{array}{ccc}
  0 & 1 & 0 \\
  0 & 0 & 1 \\
\end{array}\right)$.
A generic element of the pencil $\PP$ is thus given by 
$$A(\lambda_1,\lambda_2):=A-\lambda_1 J_1-\lambda_2 J_2=\left(%
\begin{array}{ccc}
  a_{11}-\lambda_1 & a_{12}-\lambda_2 & a_{13} \\
  a_{21} & a_{22}-\lambda_1 & a_{23}-\lambda_2 \\
\end{array}%
\right).$$

For a generic matrix $A$ the condition that the rank of $A(\lambda_1,\lambda_2)$ is less than $2$ translates into two equations: the minor consisting of the second and third columns
vanishes, and the minor consisting of the first and third
columns vanishes.
These equations have the form
\begin{eqnarray}\label{eq:minors}
  (a_{12}-\lambda_2)(a_{23}-\lambda_2)-a_{13} (a_{22}-\lambda_1) &=& 0 \label{eq:minor23}\\
  (a_{11}-\lambda_1)(a_{23}-\lambda_2)-a_{13} a_{21} &=&
  0 \label{eq:minor13}
\end{eqnarray}
Note that
\begin{equation}\label{eq:kappa}
\kappa:=\kappa_1=\frac{a_{23}-\lambda_2}{a_{13}}.
\end{equation}
Moreover, from the determinantal equation of Theorem~\ref{th:determ} we obtain a third equation. 
Substituting expression~\eqref{eq:kappa} into the latter gives the equation
\begin{equation}\label{eq:critical}
a_{13}^2 a_{11}-a_{13}^2\lambda_1+a_{23} a_{13}
a_{12}- 3 a_{23} a_{13} \lambda_2+ a_{13}
a_{23}^2-\lambda_2 a_{13} a_{12}+2
a_{13}\lambda_2^2=0.  
\end{equation}

Now equation~(\ref{eq:minor23}) has bidegree $(1,2)$ with respect to
$\lambda_1, \lambda_2$. Analogously,  (\ref{eq:minor13}) has bidegree
$(1,1)$ and (\ref{eq:critical}) has bidegree   $(1,2)$ with respect to the same variables. Clearly, 
any  solution $s$ of the system of equations consisting of (\ref{eq:minor23}), (\ref{eq:minor13}) and 
(\ref{eq:critical}) annihilates  any polynomial in the ideal generated by these
three equations. In particular, the following eight equations have  $s$ as a 
common solution:
(\ref{eq:minor23}), (\ref{eq:minor23}) multiplied by
$\lambda_2$, (\ref{eq:minor23}) multiplied by $\lambda_2^2$,
(\ref{eq:minor13}), (\ref{eq:minor13}) multiplied by
$\lambda_2$, (\ref{eq:critical}), (\ref{eq:critical}) multiplied
by $\lambda_2$, (\ref{eq:critical}) multiplied by $\lambda_2^2$. 
Therefore, the $(8\times 8)$ determinant  
\begin{equation*}
\begin{split}
&D=\\
&\det\left(%
\begin{array}{cccccccc}
\Delta_{23} & a_{13} & -a_{12}-a_{23} & 0 & 1 & 0 & 0 & 0 \\
 0 & 0 & \Delta_{23} & a_{13} & -a_{12}-a_{23} & 0 & 1 & 0 \\
 0 & 0 & 0 & 0 & \Delta_{23} & a_{13} & -a_{12}-a_{23} & 1 \\
 \Delta_{13} & -a_{23} & -a_{11} & 1 & 0 & 0 & 0 & 0 \\
0 & 0 & \Delta_{13} & -a_{23} & -a_{11} & 1 & 0 & 0 \\
\delta  & -a_{13}^2 & \sigma & 0 & 2 a_{13} & 0 & 0 & 0 \\
0 & 0 & \delta  & -a_{13}^2 & \sigma & 0 & 2 a_{13} & 0\\
0 & 0 & 0 & 0 & \delta & -a_{13}^2 & \sigma & 2 a_{13} \\
\end{array}%
\right)
\end{split}
\end{equation*}
vanishes when
(\ref{eq:minor23}), (\ref{eq:minor13}) and
(\ref{eq:critical}) have a common root.
Here we use the following notation: $\delta=a_{13}^2 a_{11}+a_{23} a_{13} a_{12}+a_{13} a_{23}^2$,
$\Delta_{23}=a_{12} a_{23}-a_{13} a_{22}$, $\Delta_{13}=a_{11}
a_{23}-a_{21} a_{13}$ and $\sigma=-a_{13} a_{12}-3 a_{13} a_{23}$.

This observation implies that the required defining polynomial for $\pi_\LL^{-1}(\C_\LL)$ is the product of some (but not necessarily all) irreducible factors of the polynomial $D$.
Factorizing $D$ we obtain $D=a_{11}^6 D_0$, where
\begin{equation*}
\begin{split}
D_0=& -12a_{13} a_{22}^2 a_{11}+a_{22}^2 a_{12}^2+12 a_{13} a_{22}
a_{11}^2+ a_{11}^2 a_{23}^2+4 a_{21}a_{12}^3\\
&-4 a_{21}a_{23}^3+a_{11}^2a_{12}^2
+12 a_{12}a_{23}^2a_{21} -12 a_{12}^2 a_{23} a_{21}-2 a_{12} a_{23}
a_{22}^2\\
&- 2 a_{12} a_{23} a_{11}^2-2
a_{22} a_{11} a_{23}^2-18 a_{13} a_{22} a_{23} a_{21}
- 2 a_{22}a_{11} a_{12}^2\\
&+18 a_{13} a_{22} a_{21} a_{12}
+18 a_{11}a_{23} a_{13} a_{21}- 18 a_{21} a_{13} a_{12} a_{11}\\
&+4 a_{13}
a_{22}^3-27 a_{21}^2 a_{13}^2-4 a_{13} a_{11}^3+ 4 a_{12} a_{23}
a_{22} a_{11}+a_{22}^2 a_{23}^2=0
\end{split}
\end{equation*}

Note that $D_0$ is of second degree in the variable $a_{13}$ and its 
discriminant (with respect to this variable) $W=
16(3a_{12}a_{21}-3a_{21}a_{23}-2a_{11}a_{22}+a_{11}^2+a_{22}^2)^3$
is not a complete square. Thus, we conclude that $D_0$ is irreducible. Hence the
variety given by $\{D=0\}$ is the union of the variety given by
$\{D_0=0\}$ and the hyperplane $\{a_{11}=0\}$ taken with multiplicity $6$.

Since the hyperplane  $\{a_{11}=0\}$ is obviously not contained in  $\pi_\LL^{-1}(\C_\LL) \subset \M(2,3)$ we obtain that $\pi_\LL^{-1}(\C_\LL)$ is given by  $\{D_0=0\}$. 
\end{example}



%
%



\begin{ack}
The authors are grateful to J.~M.~Landsberg and T.~Ekedahl for relevant discussions and to R.~Fr\"oberg for help with some of the calculations.
\end{ack}

\section{Proofs}\label{s1}

\begin{proof}[Proof of Lemma~\ref{lm:int}]
This follows almost directly from homogeneity of $\M^1$. Indeed, take any matrix $0\neq A\in \M(m,n)$. Let $\tilde l \in \LL$ be its eigenvalue, that is a matrix from $\LL$ such that $A+\tilde l$ belongs to $\M^1$.  Notice that for any $\eps\in (0,1]$ the matrix $\eps \tilde l$ is the eigenvalue of the matrix $\eps A$. Considering the family of matrices $\eps A$ with $\eps\in [0,1]$ we conclude that the
total multiplicity of eigenvalues of the pencil $A+\LL$ coincides with that of the linear pencil $\LL$ if the latter multiplicity is finite, which gives the required statement.
\end{proof}

\begin{proof}[Proof of Theorem~\ref{th:Heine}]
To get the defining system of algebraic equations for $\E_\PP$ under the assumptions of Theorem~\ref{th:Heine} we proceed exactly as in \cite{He}.  For a given
upper-triangular matrix $A\in \M(m,n)$ with distinct entries on the main diagonal  we want to find  all $(n-m+1)$-tuples $(\la_1,\ldots,\la_{n-m+1})$ such that the matrix $A+\la_1J_1+\la_2J_2+\ldots+\la_{n-m+1}J_{n-m+1}$ has positive corank. Since $A$ is upper-triangular with distinct $a_{i,i}$ then in order to get  a positive corank it is necessary to require   $\la_1+a_{i,i}=0$  for some $i=1,\ldots,m$.  The next observation is that under the above assumptions on $A$ for any given $i=1,\ldots,m$ the total number  of  eigenvalues with $\la_1+a_{i,i}=0$ equals $\binom {n-i}{m-i}$ which gives the following  count of the eigenvalues of $A$ noticed already by Heine: $\binom{n}{m-1}=\binom{n-1}{m-1}+\binom{n-2}{m-2}+\ldots+\binom{n-m}{0}.$ Indeed, if $\la_1+a_{i,i}=0$ then $\la_1+a_{j,j}\neq 0$ for all $j\neq i$ and, in particular due to the assumptions on $A$ the first $i-1$ rows of $A-a_{i,i}J_1+\la_2J_2+\ldots+\la_{n-m+1}J_{n-m+1}$ are linearly independent for all  values of $\la_2,\ldots,\la_{n-m+1}$. On the other hand, the remaining rows $i$, $i+1,\ldots,m$ can become linearly dependent under an appropriate choice of $\la_2,\ldots,\la_{n-m+1}$. Since the matrix $A-a_{1,1}J_1$ is upper-triangular with the $(i,i)$-th entry vanishing the condition that
$A-a_{i,i}J_1+\la_2J_2+\ldots+\la_{n-m+1}J_{n-m+1}$  has positive corank is equivalent to the condition that the matrix obtained by removing its first $i$ rows and $i-1$ columns has positive corank. By Lemma~\ref{lm:int} the total number of eigenvalues of the matrix of the size $(m-i+1)\times (n-i)$ equals  $\binom {n-i}{m-i}$. Let us now for any given $i=1,\ldots,m$¾ derive a system of algebraic equations in the variables $\la_2,\ldots,\la_{n-m+1}$ whose solutions are exactly  all the eigenvalues of $A$ with $\la_1+a_{i,i}=0$.  We will concentrate on the case $i=1$ since all other cases are covered in exactly the same way by working with a smaller matrix obtained from $A$ by removing the first $(i-1)$ rows and $(i-1)$ columns. Using $(k_1,\ldots,k_m)$ for the coordinates of the left kernel and $\la_1,\la_2,\ldots,\la_{n-m+1}$ for the eigenvalues we get the following system of equations
$$\begin{cases}
0=k_1(a_{1,1}+\la_1)\\
0=k_1(a_{1,2}\la_2)+k_2(a_{2,2}+\la_1) \\
...................................................\\
0=k_1(a_{1,m}+\la_m)+k_2(a_{2,m}+\la_{m-1})+...+k_m(a_{m,m}+\la_1)\\
0=k_1(a_{1,m+1}+\la_{m+1})+k_2(a_{2,m+1}+\la_{m})+...+k_m(a_{m,m+1}+\la_2)\\
......................................................................................................\\
0=k_1(a_{1,m+1}+\la_{m+1})+k_2(a_{2,m+1}+\la_{m})+...+k_m(a_{m,m+1}+\la_2)\\
0=k_1(a_{1,n}+\la_{n})+k_2(a_{2,n}+\la_{n-1})+...+k_m(a_{n,n}+\la_{n-m+1})

\end{cases}
$$
expressing the existence of a nontrivial left kernel of $A+\la_1 J_1+\la_2J_2+\ldots+\la_{n-m+1}J_{n-m+1}$. (To simplify notations we assume here that  $\la_j=0$ for $j>n-m+1$.) In order to get the required system of equations in $\la_1,\ldots,\la_{n-m+1}$ we have to eliminate from the above system the variables $k_1,\ldots,k_m$. Notice that under our assumptions on $A$ the possible corank of $A+\la_1 J_1+\la_2J_2+\ldots+\la_{n-m+1}J_{n-m+1}$ can be at most $1$ and in the case of corank $1$ the linear dependence must necessarily include the first row, i.e., $k_1=1$. Note also that the first $m$ equations are triangular with respect to $k_1,\ldots,k_m$, which together with our assumptions on $A$ allows us to successfully eliminate them. Namely, from the first equation we get $\la_1=-a_{1,1}$ and $k_1=1$. Then for any $i=2,\ldots,m$ we solve the $i$-th equation with respect to $k_i$ and get
$$k_i=\frac{1}{a_{1,1}-a_{i,i}}\left(k_1(a_{1,i}+\la_i)+k_2(a_{2,i}+\la_{i-1})+\ldots+k_{i-1}(a_{i-1,i}+\la_2)\right).$$
With the initial  value $k_1=1$ and taking into account that the only  possible denominators occurring  in the above expressions for $k_i$ are $a_{1,1}-a_{i,i}$ we recurrently find all $k_i,\; i=1,\ldots,m$ as the functions of the matrix entries and $\la$'s. Substituting these found expressions in the remaining
$n-m$ equations we get the required system  of  algebraic equations to determine $\la_2,\ldots,\la_{n-m+1}$. (Notice that $\la_1=-a_{1,1}$ was already obtained from the first equation.)
\end{proof}

\begin{example} Any matrix $A\in \M(2,4)$ has  four  eigenvalues (counted with multiplicities) with respect to  the standard diagonal subspace $\fL$. If $A$ is upper-triangular with distinct elements on the first main diagonal then these eigenvalues split into two groups  depending on the value of $\la_1$. Namely, there are $3$ eigenvalues for which $\la_1=-a_{1,1}$ and $1$ eigenvalue for $\la_1=-a_{2,2}$. For $\la_1=-a_{1,1}$ the above system (before elimination) has the form:
$$\begin{cases} 0=k_1(a_{1,1}+\la_1)\\
                              0=k_1(a_{1,2}+\la_2)+k_2(a_{2,2}+\la_1)\\
                              0=k_1(a_{1,3}+\la_3)+k_2(a_{2,3}+\la_2)\\
                              0=k_1a_{1,4}+k_2(a_{2,4}+\la_3).
\end{cases}$$
From the first equation we get $k_1=1$ and $\la_1=-a_{1,1}$. From the second equation we get
$k_2=\frac{a_{1,2}+\la_2}{a_{1,1}-a_{2,2}}$. Substituting in the remaining two equations we get the next system to determine $\la_2$ and $\la_3$:
$$\begin{cases}
(\la_2+a_{1,2})(\la_2+a_{2,3})+(a_{1,1}-a_{2,2})(\la_3+a_{1,3})=0\\
(\la_2+a_{1,2})(\la_3+a_{2,4})+(a_{1,1}-a_{2,2})a_{1,4}=0.
\end{cases}
$$
In the case  $\la_1+a_{2,2}=0$ one gets a very simple linear system:
$$k_2(a_{2,2}+\la_1)=k_2(a_{2,3}+\la_2)=k_2(a_{2,4}+\la_3)=0$$
which gives $k_2=1,\;\la_1=-a_{2,2},\;\la_2=-a_{2,3},\;\la_3=-a_{2,4}.$
\end{example}

\begin{proof}[Proof of Theorem~\ref{th:determ}]
As we already mentioned in the introduction the set $\C_\LL$ can be determined as the set of all matrices $M \in \M^1$ such that the sum of the tangent space to $\M^1$ at $M$  and the linear space $\LL$ does not coincide with the whole $\M(m,n)$. Let us describe  a basis of the tangent space to $\M^1$ at a sufficiently generic matrix $M$. Since $GL_m\times  GL_n$ acts on $\M(m,n)$ with finitely many orbits the tangent space to the $GL_m\times  GL_n$-orbit of $M$ under this action coincides with the tangent space to $\M^1$ at $M$. Note that  $GL_m\times  GL_n$ acts on $\M(m,n)$ by elementary row and column operations. Thus, if we take for example the affine chart in which the determinant formed by the first $(m-1)$ rows and columns is non-vanishing then  the tangent space to $\M^1$ at any  matrix $M$ belonging to  this chart is generated by the following two groups of operations: (i) add to each column of $M$ one of its first $m-1$ columns and (ii) add to the last row of $M$ one of its other rows. One has therefore a total of $n(m-1)+(m-1)=(n+1)(m-1)=\dim \M^1$ generators.  Taking the wedge of these generators with the chosen basis of $\LL$ and representing an $(m\times n)$-matrix as a $mn$-vector by patching together its rows we obtain the following  $(mn\times mn)$-matrix that has a block structure of an $(m\times m)$-matrix with $(n\times n)$-blocks of the form given below:

$$\mathfrak D=\begin{pmatrix}
a_{1,1}I_n& a_{2,1}I_n& \cdots&a_{m-1,1}I_n&  -\sum_{j=1}^{m-1}k_ja_{j,1}\cdot I_n\\
a_{1,2}I_n& a_{2,2}I_n& \cdots&a_{m-1,2}I_n&   -\sum_{j=1}^{m-1}k_ja_{j,2}\cdot I_n\\
\vdots        & \vdots        & \ddots & \vdots         &  \vdots   \\
a_{1,m-1}I_n& a_{2,m-1}I_n& \cdots&a_{m-1,m-1}I_n&   -\sum_{j=1}^{m-1}k_ja_{j,m-1}\cdot I_n\\
0_{m-1,n}&0_{m-1,n}&\cdots &0_{m-1,n}& \A_{m-1,n}\\
L_{1,1}    &L_{2,1}    & \cdots & L_{m-1,1} & L_{m,1}\\
L_{1,2}    &L_{2,2}    & \cdots & L_{m-1,2} & L_{m,2}\\
\vdots        & \vdots        & \ddots & \vdots         &  \vdots   \\
L_{1,n-m+1}    &L_{2,n-m+1}    & \cdots & L_{m-1,n-m+1} & L_{m,n-m+1}\\
\end{pmatrix}.
$$
Here $\A=\A_{m-1,n}=(a_{i,j})$, $i=1,\ldots,m-1$, $j=1,\ldots,n$,  $I_n$ is the identity $(n\times n)$-matrix, $0_{m-1,n}$ is the $((m-1) \times n)$-matrix with all vanishing entries, and, finally, $L_{i,j}$ is the $i$-th row of the matrix $L_j$, see Theorem~\ref{th:determ}.  Notice that the determinant $\det(a_{i,j}I_n)$, $i=1,\ldots,m-1$, $j=1,\ldots,m-1$,  of the upper-left block of $\mathfrak D$ equals $\Delta^{m-1}$, where $\Delta=\det (a_{i,j})$, $i=1,\ldots,m-1$, $j=1,\ldots,m-1$, is the leftmost principal minor of $A_{m-1,n}$.   By the above assumption the matrix $\A$ lies in the affine chart where $\Delta\neq 0$.  Finally,  we clear the low-left block $(L_{i.j})$, $i=1,\ldots,m-1$, $j=1,\ldots,n-m+1$, of $\mathfrak D$ by ``killing'' all its elements  through row operations using the above upper-left block (which is a square and non-degenerate $((m-1)n\times (m-1)n)$-matrix) to obtain the low-right  block coinciding exactly  with the matrix in formula (\ref{eq:1}). Thus the determinant of the whole matrix $\mathfrak D$ equals the product between $\Delta^{m-1}$ and the determinant from Theorem~\ref{th:determ}. Since in the considered chart one has $\Delta\neq 0$ the result follows. 
\end{proof}

\begin{proof}[Proof of Lemma~\ref{l:top}]
Set $t=i+d$, so that $t\ge 2$. Since
$$\dim \HP(i,d)=\binom{i+d-1}{d}$$
we have to show that the polynomials constructed in the lemma are linearly independent, which we prove this by induction on $t$. Note that this is trivially true for $t=2$. Assume that it holds for some $t\ge 2$ and let $i,d$ be such that $i+d=t+1$. Suppose that $c_{\al\be}\in\bC$ are such that
$$\sum_{\al\in Q_{d,d}\atop \be\in Q_{d,i+d-1}}c_{\al\be}\Big|T_{i,d}(k_1,\ldots,k_i)[\al|\be]\Big|=0.$$
Clearly, this may be rewritten as
\begin{multline}\label{eq-lin}
\sum_{\al\in Q_{d,d}\atop \be\in Q_{d,i+d-2}}c_{\al\be}\Big|T_{i,d}(k_1,\ldots,k_i)[\al|\be]\Big|\\
+k_i\cdot\!\!\!\sum_{\al\in Q_{d-1,d-1}\atop \be\in Q_{d-1,i+d-2}}c_{\al\be}\Big|T_{i,d-1}(k_1,\ldots,k_i)[\al|\be]\Big|=0.
\end{multline}
In particular, setting $k_i=0$ we get
$$\sum_{\al\in Q_{d,d}\atop \be\in Q_{d,i+d-2}}c_{\al\be}\Big|T_{i-1,d}(k_1,\ldots,k_{i-1})[\al|\be]\Big|=0$$
hence
$c_{\al\be}=0$, $\al\in Q_{d,d}$, $\be\in Q_{d,i+d-2}$,
by the induction assumption since $(i-1)+d=t$. Together with \eqref{eq-lin} this implies that
$$\sum_{\al\in Q_{d-1,d-1}\atop \be\in Q_{d-1,i+d-2}}c_{\al\be}\Big|T_{i,d-1}(k_1,\ldots,k_i)[\al|\be]\Big|=0,$$
which in turn yields $c_{\al\be}=0$, $\al\in Q_{d-1,d-1}$, $\be\in Q_{d-1,i+d-2}$, again by the induction hypothesis since $i+(d-1)=t$. We conclude that
$c_{\al\be}=0$ for all $\al\in Q_{d,d}$ and $\be\in Q_{d,i+d-1}$, which proves the desired statement hence also the lemma.
\end{proof}

\begin{proof}[Proof of Theorem~\ref{th:sds}]
We will use the setting and notation of Lemma~\ref{l:top} with $i=m$ and $d=n-m+1$. Fix the sequence $\al=\{1,\ldots,m-1\}\in Q_{m-1,n}$. Now consider the left-hand side of the determinantal equation in
Theorem~\ref{th:determ} in the case when $\fL$ is the standard diagonal subspace and $L_j=J_j$, $1\le j\le n-m+1$. In view of the generalized Laplace expansion theorem, see, e.g., \cite[\S 2.4.11]{MM}, when expanding it by the rows $\al$ this left-hand side becomes
\begin{multline*}
(-1)^{\frac{m(m-1)}{2}}\sum_{\be\in Q_{m-1,n}}(-1)^{\rho(\be)}\Big|\A\big[\{1,\ldots,m-1\}|\be\big]\Big|\times\\
\Big|T_{m,n-m+1}(k_1,\ldots,k_m)\big[\{1,\ldots,n-m+1\}|\{1,\ldots,n\}\setminus\be\big]\Big|,
\end{multline*}
which proves the theorem.
\end{proof}



\section{Remarks and open questions}

\subsection*{A}
By analogy with the  above case, for a given triple $n,m,r$ one can   also consider $(m-r)(n-r)$-dimensional pencils of matrices in $\M(m,n)$ and study their intersections with the subvariety $\M^r$ of all matrices of corank at least $r$. In particular, a natural question is to find an analog of Theorem~\ref{th:determ} in this situation.

\subsection*{B}
It would  be interesting to determine the equation for $\pi^{-1}_\LL(\C_\LL)$ in general, see Example 3 in the Introduction. Another important direction is to determine the local multiplicity of a given eigenvalue in terms of the defining polynomial for $\C_\LL$.   Is there any analog of the Jordan normal form allowing to determine the multiplicity of a given eigenvalue?

\subsection*{C}
Notice that the left-right action of $GL_m\times GL_n$ extends from the space $\M(m,n)$ to every space of (in)complete flags in $\M(m,n)$. For simple dimensional reasons, in most cases this action cannot have finitely many orbits.

\begin{problem}
On which spaces of (in)complete flags  the above   left-right action of $GL_m\times  GL_n$ has finitely many orbits?
\end{problem}


\begin{thebibliography}{9999999}

\bibitem[AVG]{AVG} V.~Arnold, A.~Varchenko, S.~Gusein-Zade,
 Singularities of differentiable maps. Vol. I. The classification of critical points, caustics and wave fronts.
Monogr. Math. {\bf 82}, Birkh\"auser Boston, Inc., Boston, MA, 1985.

\bibitem[BBS]{BBS}
J.~Borcea, P.~Br\"anden, B.~Shapiro, {\em Algebraic and geometric aspects of Heine-Stieltjes theory }, in preparation.

\bibitem[BEGM]{BEGM} G.~Boutry,  M.~Elad, G.~ Golub,  P.~Milanfar,  {\em The generalized eigenvalue problem for nonsquare pencils using a minimal perturbation approach}, SIAM J. Matrix Anal. Appl. {\bf 27} (2005), 582--601.

\bibitem[BV]{BV} W.~Bruns,  U.~Vetter,  Determinantal rings. Lect. Notes Math. {\bf 1327}, Springer-Verlag, Berlin, 1988.

\bibitem [CG]{CG} D.~Chu, G.~Golub, {\em  On a generalized eigenvalue problem for nonsquare pencils.} SIAM J. Matrix Anal. Appl. {\bf 28} (2006), 770--787.

\bibitem [He]{He} E.~Heine, Handbuch der Kugelfunctionen. Vol.1,
pp.~472--479, Berlin: G. Reimer Verlag, 1878.

\bibitem[MM]{MM}
M.~Marcus, H.~Minc, A survey of matrix theory and matrix inequalities. Allyn and Bacon, Inc., Boston, MA, 1964.

\bibitem[TW]{TW} L.~Trefethen, T.~Wright,  {\em  Pseudospectra of rectangular matrices}, IMA J. Numer. Anal. {\bf 22} (2002), 501--519.

\bibitem [Vol]{Vol}H.~Volkmer, Multiparameter eigenvalue problems and
expansion theorems. Lect. Notes. Math. {\bf 1356}, Springer-Verlag, 1988.

\end{thebibliography}
\end{document}